\documentclass[11pt,twoside]{article}

\usepackage{mathrsfs,amsfonts,amsmath,amssymb}
\usepackage{latexsym}
\newtheorem{satz}{Theorem}

\newtheorem{theorem}[satz]{Theorem}
\newtheorem{lemma}[satz]{Lemma}

\newtheorem{remark}[satz]{Remark}

\def\no{\noindent}
\def\sbeq{\subseteq}

\def\R{\mathbb {R}}

\def\a{\alpha}

\def\a{\alpha}

\def\o{\omega}
\def\({\big (}
\def\){\big )}

\def\D{\Delta}

\begin{document}

\title{A remark on  $A+B$ and $A-A$ for compact sets in $\R^n$}
\author{ By\\  \\{\sc Tomasz Schoen\footnote{\textsc{The author is supported by NCN grant 2012/07/B/ST1/03556.}} ~~and~~ Ilya D. Shkredov
}}
\date{}
\maketitle

\begin{abstract} We prove in particular that if $A\subset \R^n$ be a compact convex set, and $B\subset \R^n$ be an arbitrary compact set then $\mu (A-A) \ll \frac{\mu(A+B)^2}{\sqrt{n} \mu (A)}$, provided that $\mu(B)\ge \mu(A)$.
\end{abstract}

A well--known Ruzsa triangle inequality states that for any finite sets of an abelian group we have
$$|A-B|\le \frac{ |A+C||C+B|}{|C|} \,,$$
in particular if $B=A$ and $C=B$, then
$$ |A-A|\le \frac{ |A+B|^2}{|B|} \,.$$
The aim of this note is to prove a sharp up to a dimension--independent constant form of the above inequality for  a compact convex set $A\subset \R^n$, and  an arbitrary compact set $B\subset \R^n$, provided that $\mu(A)\ge \mu(B)$.
\bigskip

For a set  $A \subset \R^n$ and $x\in A-A$ put
$$
    A_x =  A\cap (A-x) \,.
$$

\noindent Our main tool is the following lemma proved in \cite{s_balog} (Lemma 5). We recall its proof as it is very simple.

\begin{lemma}
    Let $A,B \subset \R^n$ be compact sets.
    Then
\begin{equation}\label{f:A^2+D(A)}
    \int_{A-A} \mu ( A_x+B)\, dx \le \mu (A+B)^2 \,.
\end{equation}\label{l:A^2+D(A)}
\end{lemma}
\begin{proof}
We apply a well-known Koester-Katz transform: if $x\in A-A$ then
$$A_x+B\sbeq (A+B)_x \,.$$
Therefore, we  have
$$\int _{A-A}\mu(A_x+B)dx\le \int _{A+B-A-B}\mu((A+B)_x)dx=\mu(A+B)^2 \,,$$
and the assertion follows.
$\hfill\Box$
\end{proof}

\bigskip

\noindent We also need a lower bound for  the size of $A_x$ for a convex set $A$, see \cite{TV} section 3.
We also give the proof for the sake of completeness.

\begin{lemma}
    Let $A\subset \R^n$ be a compact convex set and $r\in [0,1]$ be any real number.
    Then for all $x\in r(A-A)$
    the following holds
\begin{equation}\label{f:convex_A_x}
    \mu(A_x) \ge (1-r)^n \mu (A) \,.
\end{equation}
\label{l:convex_A_x}
\end{lemma}
\begin{proof}
    Write  $x=ra_1-ra_2$, where $a_1,a_2\in A$ and let
     $a\in A$ be any element.
    By convexity $(1-r) a + ra_1 \in A$ and $(1-r) a + ra_1=(1-r) a + ra_2+x \in A+x$.
    Thus $(1-r) A + ra_1 \subseteq A \cap (A+x)$
    and the result follows.
$\hfill\Box$
\end{proof}

\bigskip
\noindent Finally, we recall the Brunn-Minkowski inequality, see \cite{TV} section 3.

\begin{theorem}
    Let $A,B\subset \R^n$ be non--empty compact sets.
    Then
\begin{equation*}\label{f:B-M}
    \mu (A+B)^{1/n} \ge \mu (A)^{1/n} + \mu (B)^{1/n} \,.
\end{equation*}
\label{t:B-M}
\end{theorem}
\noindent Now we can formulate our main result.

\begin{theorem}
    Let $A\subset \R^n$ be a compact convex set, and $B\subset \R^n$ be an arbitrary compact set.
    Then
\begin{equation}\label{f:A+A,A-A-}
    (1+\o + \dots + \o^{[\sqrt{n}]}) \mu (B)^{1-1/n} \mu (A)^{1/n} \mu (A-A) \ll \mu (A+B)^2 \,,
\end{equation}
    where $\o = (\mu (A) / \mu (B))^{1/n}$.
    In particular, if $\mu (A) \ge \mu (B)$ then
\begin{equation}\label{f:A+A,A-A}
    \mu (A-A) \ll \frac{\mu (A+B)^2}{\sqrt{n} \mu (A)^{1/n} \mu (B)^{1-1/n}}
    \,,
\end{equation}
    and if $\mu (B) \ge \mu (A)$ then
\begin{equation}\label{f:A+A,A-A+}
    \mu (A-A) \ll \frac{\mu(A+B)^2}{\sqrt{n} \mu (A)}
    \,.
\end{equation}
\label{t:A+A,A-A}
\end{theorem}
\begin{proof}
    Let $\a = \mu (B) / \mu (A)$.
    Applying (\ref{f:A^2+D(A)}) and the Brunn-Minkowski inequality,
    we get
\begin{eqnarray*}
    \mu^{2} (A+B) &\ge& \int_{A-A} \mu (B + A_x)\, dx
        \ge
            \int_{A-A} \left( \mu (B)^{1/n} + \mu (A_x)^{1/n} \right)^n \, dx\\
    &=&
        \a \sum_{k=0}^n \binom{n}{k} \int_{A-A} \a^{-k/n} \mu (A)^{(n-k)/n} \mu (A_x)^{k/n}\, dx \,.
\end{eqnarray*}
    To estimate the size of $A_x$ we  use Lemma \ref{l:convex_A_x}.
    After integration by parts, we obtain
\begin{eqnarray*}
    \mu^{2} (A+B) &\ge& \mu (B) \sum_{k=0}^n \binom{n}{k} k \a^{-k/n} \int_0^1 (1-r)^{k-1} \mu(r (A-A)) \, dr\\
        &=&\mu (B) \mu (A-A) \sum_{k=1}^n \binom{n}{k} k \a^{-k/n} \int_0^1 (1-r)^{k-1} r^n \, dr\\
               &=& \mu (B) \mu (A-A) \sum_{k=1}^n \binom{n}{k} k\a^{-k/n} {\cal B}(k,n+1)  \,,
\end{eqnarray*}
where ${\cal B}(\cdot,\cdot)$ is the beta function.
Thus
$$
    \mu^{2} (A+B) \ge \mu (B) \mu (A-A) \sum_{k=1}^n \a^{-k/n} \frac{(n!)^2}{(n-k)! (n+k)!}
        :=
            \mu (B) \mu (A-A) \times \sigma \,.
$$
One can calculate the last sum $\sigma$ using the gamma function or hypergeometric series,  but we use a rather crude estimate.
Put
$\D = [\sqrt{n}]+1$,
then
\begin{eqnarray*}
    \sigma &=& \sum_{k=1}^n \a^{-k/n} \prod_{j=1}^{k-1} \left(1-\frac{j}{n} \right) \prod_{j=1}^k \left(1+\frac{j}{n} \right)^{-1}
        =
            \sum_{k=1}^n \a^{-k/n} \left(1+\frac{k}{n} \right)^{-1} \prod_{j=1}^{k-1} \left(1-\frac{2j}{n+j} \right)
                \,.
\end{eqnarray*}
Using inequalities $\ln(1-x) \ge -2x$ for $0\le x \le 0.5$ and $k\le n$, we obtain 
\begin{eqnarray*}
\sigma
    &\ge&
        \frac12 \sum_{k=1}^{\D} \a^{-k/n} \exp \left( -\sum_{j=1}^{k-1} \frac{4j}{n+j} \right)
            \ge
                \frac12 \sum_{k=1}^{\D} \a^{-k/n} \exp \left( -\frac{2k^2}{n} \right)
                    \gg
                        \sum_{k=1}^{\D} \o^k \,.
                        \end{eqnarray*}
This gives us (\ref{f:A+A,A-A-}).
                        To see (\ref{f:A+A,A-A}) it is enough to observe that if $\mu(A)\ge \mu (B)$ then
$\sum_{k=1}^{\D} \o^k\ge \sqrt{n}$.
To get (\ref{f:A+A,A-A+}) take any subset $B'$ of $B$ such that $\mu (B')=\mu(A)$ and apply
 (\ref{f:A+A,A-A}), then
 $$ \mu (A-A) \ll \frac{\mu(A+B')^2}{\sqrt{n} \mu (A)} \le \frac{\mu(A+B)^2}{\sqrt{n} \mu (A)} \,.$$
This completes the proof.
$\hfill\Box$
\end{proof}

\begin{remark}
    Estimate (\ref{f:A+A,A-A}) is tight, see paper \cite{RS} or book \cite{Ruzsa_book}, discussion after Corollary 8.3.
    Indeed,
    consider
    $n$--dimensional simplex
$$
    A  = A_L = \{ (x_1,\dots,x_n) \in \R^n ~:~ x_j \ge 0,\, \sum_{j=1}^n x_j \le L \} \,,
$$
    where $L$ is a parameter.
    Then $\mu(A+A) = 2^n \mu (A)$ and 
$ 
        \mu (A-A) = \binom{2n}{n} \mu (A) 
$
    (to obtain the last formula one can count to number of  integer points in $A$, say, and approximate $\mu (A-A)$ by 
    $$
        \sum_{a+b+c = n} \binom{n}{a,b,c} \binom{L}{a} \binom{L}{b} \sim \frac{L^n}{n!} \sum_{m=0}^n \binom{n}{m}^2 = \frac{L^n}{n!} \binom{2n}{n} = \mu (A) \binom{2n}{n} \,,
    $$ 
    see \cite{Granville}. Here $a,b,c$ the number of possibilities for the positive, negative and zero coordinates in $A-A$, correspondingly).
    Hence
$$
    \mu (A-A) \gg \frac{\mu (A+A)^2}{\sqrt{n} \mu (A)} \,.
$$
\end{remark}

{}
\bigskip

\noindent{T.~Schoen\\
\no{Faculty of Mathematics and Computer Science,\\ Adam Mickiewicz
University,\\ Umul\-towska 87, 61-614 Pozna\'n, Poland\\} {\tt
schoen@amu.edu.pl}
\bigskip

\noindent{I.D.~Shkredov\\
Steklov Mathematical Institute,\\
ul. Gubkina, 8, Moscow, Russia, 119991}
\\
and
\\
IITP RAS,  \\
Bolshoy Karetny per. 19, Moscow, Russia, 127994\\
{\tt ilya.shkredov@gmail.com}


\begin{thebibliography}{99}


\bibitem{Granville}
{\sc A. Granville, }
{\em An introduction to additive combinatorics, }
CRM Proceedings and Lecture Notes 43 (2007), 1--27.


\bibitem{RS} {\sc C. A. Rogers, G. C. Shephard, }
{\em The difference body of a convex body, }
Arch. Math 8 (1957), 220--233.


\bibitem{Ruzsa_book} {\sc I. Z. Ruzsa, }
{\em Sumsets and structure, } Combinatorial number theory and additive group theory, Adv. Courses Math. CRM Barcelona, Birkh¨auser Verlag, Basel, 2009, 87--210.




\bibitem{s_balog}
{\sc I. D. Shkredov, }
{\em On a question of A. Balog, }
Pacific Journal of Mathematics, (2015); dx.doi.org/10.2140/pjm.2015..101; arXiv:1501.07498v1 [math.CO] 29 Jan 2015.



\bibitem{TV}
{\sc T. Tao, V. Vu, }
{\em Additive Combinatorics, }
Cambridge University Press (2006).


\end{thebibliography}
\end{document}